\documentclass{article}
\usepackage{amsfonts}

\usepackage{graphicx}
\usepackage{amsmath}
\usepackage[symbol]{footmisc}

\newtheorem{theorem}{Theorem}[section]

\newtheorem{corollary}[theorem]{Corollary}

\newtheorem{definition}{Definition}

\newtheorem{lemma}[theorem]{Lemma}

\newtheorem{proposition}[theorem]{Proposition}
\newtheorem{remark}{Remark}

\newenvironment{proof}[1][Proof]{\textbf{#1.} }{\ \rule{0.5em}{0.5em}}
\input{tcilatex}

\begin{document}

\begin{center}
{\huge A 2-variable power series approach to the Riemann hypothesis }

\bigskip

{\huge \bigskip}

{\large Vincent Brugidou}$^{a,\,b}$ 
\footnotetext{\textit{E-mail address:}vincent.brugidou@univ-lille1.fr}
\end{center}

\bigskip

\begin{center}
$^{a}$\textit{Laboratoire Paul Painlev\'{e}, Universit\'{e} de Lille 1,
59655 Villeneuve d'Ascq cedex, France}

$^{b}$\textit{IUT A de Lille1, Bld Paul Langevin, BP 179, 59653 Villeneuve
d'Ascq cedex, France}
\end{center}

\bigskip

\textbf{Abstract: }We consider the power series in two complex variables $%
B_{y}\left( f^{\flat }\right) \left( x\right) =\sum_{n\geq 0}A_{n}^{\flat
}x^{n}y^{n\left( n+1\right) /2}$, where $\left( -1\right) ^{n}A_{n}^{\flat }$
are the non-zero coefficients of the Maclaurin series of the Riemann Xi
function. The Riemann hypothesis is the assertion that all zeros of $%
B_{1}\left( f^{\flat }\right) $ are real. We prove that every zero of $%
B_{y}\left( f^{\flat }\right) $ is the inverse of a power series in $y$ with
real coefficients, which converges for $\left| y\right| <0,2078...$ . We
show the existence of a constant $\Theta $, similar to the de Bruijn-Newman
constant, satisfying : $0\leq $ $y$ $\leq \Theta $ if and only if all zeros
of $B_{y}\left( f^{\flat }\right) $ are real. We prove that $1/4\leq \Theta
\leq 1$ and that $\Theta =1$ is equivalent to the Riemann hypothesis. We
show that the Riemann hypothesis is equivalent to what the discriminant of
each Jensen polynomial of $B_{y}\left( f^{\flat }\right) $ does not vanish
on the interval $\left[ 1/4,1\right[ $. We prove the Riemann hypothesis
implies that the zeros of $B_{y}\left( f^{\flat }\right) $ are simple for $%
0<y<1$, and we conjecture that the reciprocal implication is true.

\textit{keywords: }Riemann hypothesis, two complex variables, Laguerre
entire functions, de Bruijn-Newman constant, analytic continuation, simple
zeros.

\section{\protect\bigskip \textbf{Introduction}}

Following standard practice, we define the Riemann Xi function as 
\begin{equation}
\Xi \left( z\right) =\dfrac{1}{2}s\left( s-1\right) \pi ^{-s/2}\Gamma \left(
s/2\right) \zeta \left( s\right) \text{ \ ; \ \ \ }s=\dfrac{1}{2}+iz
\end{equation}
where $\zeta $ is the Riemann zeta function. It is well known ($\left[ 15%
\right] $ chap.1, 10) that $\Xi $ is an entire real function of order 1 and
that the Riemann hypothesis (abbreviated RH) is the assertion that all zeros
of $\Xi $\ are real. We have 
\begin{equation}
\Xi \left( z\right) =\sum\limits_{n=0}^{\infty }\left( -1\right)
^{n}A_{n}^{\flat }z^{2n}
\end{equation}
the coefficients $A_{n}^{\flat }$ being some strictly\ positive numbers ($%
\left[ 9\right] $, p.41). Introduce the auxiliary function

\begin{equation}
f^{\flat }\left( x\right) =\Xi \left( -i\sqrt{x}\right) =\sum_{n=0}^{\infty
}A_{n}^{\flat }x^{n}
\end{equation}
$f^{\flat }$ is an entire real function of order $1/2$ and factorizes
therefore, according to the Hadamard theorem, as

\begin{equation}
f^{\flat }\left( x\right) =A_{0}^{\flat }\prod\limits_{p\geq 1}^{\infty
}\left( 1-\dfrac{x}{x_{p}}\right)
\end{equation}
Thus, RH is also the assertion that all the zeros $x_{p}$ of $f^{\flat }$
are negative numbers (in fact, it suffices to show that the zeros $x_{p}$
are real since the coefficients $A_{n}^{\flat }$ are positive).

In\ $\left[ 3\right] $ a new method was introduced by the author, to
determine the value and the reality of zeros of an entire function, whose
all coefficients of its Maclaurin series are non-zero real numbers (this
method has been generalized in $\left[ 4\right] $ to the functions which may
be represented by a Laurent series convergent in $\mathbb{C}^{\ast }$). It
was therefore natural to apply this method to the function $f^{\flat }$ that
satisfies these characteristics ( note that this is why we introduce the
auxiliary function $f^{\flat }$, rather than using the function $\Xi $ that
does not chek these characteristics). This is the purpose of this article,
whose here is the plan and the main results. In Section 2, we recall the
principle of the method introduced in $\left[ 3\right] $, in defining a
functionnal transformation called $y$-Borel transform and denoted by $B_{y}$%
. By direct application of Theorem 16 of $\left[ 3\right] $, we show that
the inverse of zeros of $B_{y}\left( f^{\flat }\right) $ may be represented
by power series in $y$, convergent in a disk with center $0$ and radius
equal to $\rho _{o}^{2}=0,2078...$ where $\rho _{o}$ is the positive root of
the equation $\sum_{k=1}^{\infty }\rho ^{k^{2}}=1/2$ (\textbf{Theorem 2.2})
. In the following sections, we limit ourselves to study $B_{y}\left(
f^{\flat }\right) $ for $y\in \mathbb{R}_{+}$. In Section 3, we recall some
properties of the Laguerre entire functions, necessary for the rest of the
article. In Section 4, we define from $B_{y}\left( f^{\flat }\right) $ a
constant $\Theta $, whose \textbf{Theorem 4.3}\ shows that it plays a role
similar to the de Bruijn-Newman constant. In particular, the point $iv)$ of
Theorem 4.3 is a reformulation of RH. In Section 5, based on the results
already established, we get another reformulation of RH by using the theory
of analytic continuation along a path (\textbf{Theorem 5.1}). In Section 6,
we prove (\textbf{Corollary 6.2}) that the Riemann hypothesis implies that
the zeros of $B_{y}\left( f^{\flat }\right) $ are simple for $y\in \left] 0,1%
\right[ $ and we discuss the reciprocal implication.

\bigskip

\begin{remark}
It is easy to see that RH is equivalent to what $f^{\flat }$ is a Laguerre
entire function (see Section 3 below for the definition of these functions).
In fact, many of the necessary or sufficient conditions for the RH, that we
exhibit in this article, follow from conditions of which we will demonstrate
the necessity or the sufficiency, in order that an entire real function may
be a Laguerre entire function. So, our article can also be considered from
this point of view. We can say that our approach to the RH is an ''external
approach'', in the words of Balazard in Section 3 of $\left[ 1\right] $. The
theory of the Fourier transforms with real zeros (see $\left[ 1\right] $
Section 3.2 with, in particular, Polya's results), is a well known
''external approach'' of RH. These two ''external approaches'' present
similarities as we will see below in Section 4.3.
\end{remark}

\bigskip\ 

\section{\protect\bigskip Direct application of the method introduced in $%
\left[ 3\right] $}

\bigskip To present this method, it is convenient to introduce a functional
transformation:

\subsection{Formal $y$-Borel transform}

\begin{definition}
Let $\mathbb{K}$ be a field and $f\left( x\right) =\sum_{n\in \mathbb{N}}$ $%
A_{n}x^{n}$ $\in \mathbb{K}\left[ \left[ x\right] \right] $. We will call $y$%
-Borel transform of $f$, the formal power series in $x$ and $y$ 
\begin{equation}
B_{y}\left( f\right) \left( x\right) =Q\left( x,y\right) =\sum\limits_{n\in 
\mathbb{N}}A_{n}x^{n}y^{\tfrac{n\left( n+1\right) }{2}}
\end{equation}
\end{definition}

This definition is still valid when $y$ is a fixed element of $\mathbb{K}$, $%
B_{y}\left( f\right) $ is then a formal power series in $x$ depending on the
parameter $y$. Note that we have in particular $B_{1}\left( f\right) =f$.
When $\mathbb{K=C}$, these definitions and notations extend naturally to
cases of convergent power. In the following, we use the notation $%
B_{y}\left( f\right) $ in general, but it may happen that we use the
notation $Q\left( x,y\right) $ when we consider the $y$-Borel transform as a
function of two variables.

If we change variables $y=1/q$, $x=\xi q$ \ ( here $\xi $ is a simple
variable) in the equation $\left( 5\right) $, $B_{y}$ reduces to a $q$%
-analogue of the formal Borel transform, introduced by Ramis (see eg $\left[
16\right] $, Definition 4.4). The relationship is clearly between $B_{y}$
and the $q$-Borel transform as in Ramis assume $\left| q\right| \geq 1$,
whereas for us the interesting case is $\left| y\right| \leq 1$. Note
however, that our results are independant of the theory of $q$-difference
equations, for which the $q$-formal Borel transform was introduced by Ramis.

Note also that we have the following relation valid for $y\in \mathbb{K}$
and $y_{0}\in \mathbb{K}^{\ast }$ 
\begin{equation}
B_{y}\left( f\right) =B_{y/y_{0}}\left[ B_{y_{0}}\left( f\right) \right]
\end{equation}
and the following convergence properties valid for $\mathbb{K=C}$:

\begin{proposition}
If for $y_{0}\in \mathbb{R}_{+}^{\ast }$, $B_{y_{0}}\left( f\right) $ is an
entire function, then $B_{y}\left( f\right) $ is an entire function for $%
\left| y\right| \leq y_{0}$. Furthermore, when $y\rightarrow y_{0}$, $\
0\leq y<y_{0}$, $B_{y}\left( f\right) $ converges uniformly to $%
B_{y_{0}}\left( f\right) $ on every compact subsets of $\mathbb{C}$, and the
zeros of $B_{y_{0}}\left( f\right) $ are the limits of zeros of $B_{y}\left(
f\right) $.
\end{proposition}

\begin{proof}
\bigskip Suppose first that $y_{0}=1$. Then, it is a special case of Lemma
19 of $\left[ 3\right] $. So, we summarize the proof briefly. Let $R$\ be
any positive number, for $\left| x\right| \leq R$ and $y\in \overline{D}%
\left( 0,1\right) $ $=\left\{ y\in \mathbb{C}\text{ : }\left| y\right| \leq
1\right\} $, we have $\left| A_{n}x^{n}y^{n\left( n+1\right) /2}\right| \leq
\left| A_{n}\right| R^{n}$, then the power series $\sum A_{n}x^{n}y^{n\left(
n+1\right) /2}$ converges normally with respect to $\left| x\right| \leq R$
and $\left| y\right| \leq 1$. $Q\left( x,y\right) $ is therefore a
continuous function in $\mathbb{C}\times \overline{D}\left( 0,1\right) $. It
follows that $Q\left( x,y\right) $ is uniformly continuous on each $K\times 
\overline{D}\left( 0,1\right) $, where $K$ is an arbitrary compact of $%
\mathbb{C}$. We deduce that $B_{y}\left( f\right) $ is an entire function
for $\left| y\right| \leq 1$ and $B_{y}\left( f\right) $ converges to $%
B_{1}\left( f\right) =f$ uniformly with respect to $x$\ on each compact
subsets of $\mathbb{C}$, when $y\rightarrow 1$, $y$ remaining in $\overline{D%
}\left( 0,1\right) $. Then we prove that, similarly to the Hurwitz theorem,
the zeros of $f$ are the limits of zeros of $B_{y}\left( f\right) $ when $%
y\rightarrow 1$, $y$ remaining in $\overline{D}\left( 0,1\right) $. This is
a fortiori true when $y\longrightarrow 1$ along the radius $\left[ 0,1\right[
$.

For any $y_{0}\in \mathbb{R}_{+}^{\ast }$, we reduce to the previous case
using the relationship $\left( 6\right) $.
\end{proof}

\subsection{\protect\bigskip Summary of the method introduced in $\left[ 3%
\right] $}

Let $\mathbb{K}$ be a field of characteristic zero and $f\left( x\right)
=\sum_{n\geq 0}A_{n}x^{n}$ $\in $ $\mathbb{K}\left[ \left[ x\right] \right] $
with $A_{n}\neq 0$ for all $n$. By considering $B_{y}\left( f\right) $ as a
formal power series in $x$ over the field $\mathbb{K}\left( \left( y\right)
\right) $, it was shown (Theorem 1 of $\left[ 3\right] $) that the zeros of $%
B_{y}\left( f\right) $ are the elements of\ $\mathbb{K}\left( \left(
y\right) \right) $, given for $p\geq 1$ by 
\begin{equation}
x_{p}\left( y\right) =-\dfrac{1}{\alpha _{p}\left( y\right) }
\end{equation}
where $\alpha _{p}\left( y\right) \in \mathbb{K}\left[ \left[ y\right] %
\right] $ is given by 
\begin{equation}
\left\{ 
\begin{tabular}{ll}
\ \ $\alpha _{p}\left( y\right) $ & $=\sum\limits_{q\geq p}u_{p}\left(
q\right) y^{q}$ \\ 
and$\text{ \ \ \ \ \ \ }u_{p}\left( q\right) $ & $=v_{p}\left( q\right)
-v_{p+1}\left( q\right) $%
\end{tabular}
\right. 
\end{equation}
in particular 
\begin{equation}
u_{p}\left( p\right) =\dfrac{A_{p}}{A_{p-1}}.
\end{equation}

Further, Theorem 15 of $\left[ 3\right] $ shows that for $p\geq 1$, $q\geq p$

\begin{equation}
v_{p}\left( q\right) =\left[ x^{1}y^{q-\left( p-1\right) }\right]
LogQ_{p}\left( x,y\right)
\end{equation}
with

\begin{equation*}
Q_{p}\left( x,y\right) =\dfrac{y^{\left( p-2\right) \left( p-1\right) /2}}{%
A_{p-1}x^{p-1}}Q\left( x,y\right) .
\end{equation*}
In the equation $\left( 10\right) $, $LogQ_{p}$ denotes the extended
composition, as it was defined in Section 4 of $\left[ 3\right] $,\ of the
formal power series $Q_{p}\left( x,y\right) \in \mathbb{K}\left( \left(
x\right) \right) \left[ \left[ y\right] \right] $ by $Log\left( 1+z\right)
\in \mathbb{K}\left[ \left[ Z\right] \right] $. And we use the notation $%
\left[ x^{i}y^{j}\right] $ to denote the coefficient in $x^{i}y^{j}$ in the
2-variable power series which follows this symbol. We thus get $v_{p}\left(
q\right) $ (Section 5 of [3]), as a product of $A_{1}/A_{0}$ by a polynomial
expression in the ''variables'' $\Omega _{n}$ defined for $n\geq 1$, by 
\begin{equation*}
\;\Omega _{n}=\dfrac{A_{n-1}A_{n+1}}{A_{n}^{2}}\text{.}
\end{equation*}
Suppose now that $\mathbb{K=C}$. When $R>0$, we say that $f$ satisfies the
property $A\left( R\right) $ if for all real number $r$, $0\leq r<R$, we have

\begin{equation}
\text{\ \ }\sum\limits_{p\geq 1,q\geq p}v_{p}\left( q\right) r^{q}<+\infty 
\text{ .}
\end{equation}
We then put down $R^{\ast }=\sup \left\{ R>0\;/\text{ }A\left( R\right) 
\text{ is true}\right\} $ if this set is not empty, and $R^{\ast }=0$
otherwise.

It has been shown (Theorem 3 of\ $\left[ 3\right] )$ for $y\in \mathbb{C}$
such that $\left| y\right| <R^{\ast }$, that the power series $\alpha
_{p}\left( y\right) \ $converge as well as the infinite product 
\begin{equation}
B_{y}\left( f\right) \left( x\right) =A_{0}\prod\limits_{p\geq 1}\left(
1+\alpha _{p}\left( y\right) x\right) \text{.}
\end{equation}
The zeros of $B_{y}\left( f\right) $ are therefore exactly given by Equation 
$\left( 7\right) $ for $\left| y\right| <R^{\ast }$, $p$ taking all values
such that $\alpha _{p}\left( y\right) \neq 0$.

Using a method of majorant series, it was also shown ( Theorem 16 of $\left[
3\right] $) that if we put $\Omega =\sup_{n\geq 1}\left| \Omega _{n}\right| $%
, we have $R^{\ast }>\rho _{o}^{2}\Omega ^{-1}$where $\rho _{o}$ is the
positive root of the equation in $\rho $ 
\begin{equation}
\sum_{k=1}^{\infty }\rho ^{k^{2}}=1/2.
\end{equation}

In particular, if all coefficients $A_{n}$ are real and if $\Omega \leq \rho
_{o}^{2}=0,2078...$, all the zeros of the entire function $f$ are real
(corollary 18 of $\left[ 3\right] $).

Note that if $f$ is a polynomial of degree $n$ whose all coefficients are
different from zero, the method outlined above is fully applicable for the $%
n $ zeros $x_{p}\left( y\right) $, $1\leq p\leq n$, with the following
convention : $\Omega _{k}=0$ for $k\geq n$. In $\left[ 4\right] ,$ these
results were extended to the case of Laurent series $\sum_{n\in \mathbb{Z}%
}A_{n}x^{n}$, convergent on $\mathbb{C}^{\ast }$.

Finally, note that the factorization of $Q\left( x,y\right) $ in infinite
product and the formula $\left( 10\right) $ can be generalized to the case
where : 
\begin{equation*}
Q\left( x,y\right) =\sum_{n}A_{n}\left( y\right) x^{n}y^{n\left( n+1\right)
/2}
\end{equation*}

where $A_{n}\left( y\right) $ is now a series in $y$ with a valuation equal
to zero, $\left[ 5\right] $ . This latter result can be easily applied to
factorize a large number of $q$-series in an infinite product.

\subsection{Application to the function $f^{\flat }$}

As noted in the introduction, the coefficients $A_{n}^{\flat }$ of the
entire function $f^{\flat }$ are real and different from zero. We can
therefore apply the method described above. This gives the following new
theorem:

\begin{theorem}
\bigskip Let $\rho _{o}$ be the positive root of equation $\left( 13\right) $%
\ and $f^{\flat }$ be the entire function defined by equation $\left(
3\right) $. Then the zeros of the $y$-Borel transform of $f^{\flat }$can be
calculated explicitly by the formulas $\left( 7\right) ,\left( 8\right)
,\left( 10\right) $ for $0<\left| y\right| <\rho _{o}^{2}=0,2078...$ And
they are all real for $0<y\leq \rho _{o}^{2}$.
\end{theorem}

\begin{proof}
\bigskip In $\left[ 8\right] $ (Csordas and al,1986), it was proved for $%
n\geq 1$, the double inequality 
\begin{equation*}
1\leq \dfrac{b_{n-1}b_{n+1}}{b_{n}^{2}}\leq \dfrac{2n+1}{2n-1}
\end{equation*}
where the coefficients $b_{n}$ are related to our coefficients $A_{n}^{\flat
}$ by $b_{n}/(2n)!=A_{n}^{\flat }/2^{n+3}$. It follows that 
\begin{equation}
\left( \dfrac{2n-1}{2n+1}\right) \dfrac{n}{n+1}\leq \Omega _{n}^{\flat }=%
\dfrac{A_{n-1}^{\flat }A_{n+1}^{\flat }}{\left( A_{n}^{\flat }\right) ^{2}}%
\leq \dfrac{n}{n+1}
\end{equation}
(the inequality on the right side of $\left( 14\right) $ is sometime called
the Turan inequality). Proof is now a direct application of Theorem 16 of $%
\left[ 3\right] $ to the function $f^{\flat }$ with $\Omega =\sup_{n\geq
1}\left| \Omega _{n}^{\flat }\right| =1$ and gives $R^{\ast }\geq \rho
_{o}^{2}$. So, for $\left| y\right| <$ $\rho _{o}^{2}$, the power series $%
\alpha _{p}\left( y\right) $ converge for all $p\geq 1$, and we get all the
zeros of $B_{y}\left( f^{\flat }\right) $ with $\left( 7\right) $ by taking
all values of $p\geq 1$ such as $\alpha _{p}\left( y\right) \neq 0$.
Further, as the coefficients $A_{n}^{\flat }$ are real numbers, it is the
same for $v_{p}\left( q\right) $ then for $\alpha _{p}\left( y\right) $, and
therefore for all the zeros of $B_{y}\left( f^{\flat }\right) $ when $y\in %
\left[ 0,\rho _{o}^{2}\right[ $. It follows from proposition 2.1\ that the
zeros of $B_{y}\left( f^{\flat }\right) $ are also real numbers when $y=\rho
_{o}^{2}$
\end{proof}

Assume now we can show that $R^{\ast }\geq 1$ for $f^{\flat }$. Then the
results of Theorem 2.2 extend to the disk $\left| y\right| <1$.
Specifically, Theorem 16 of $\left[ 3\right] $ shows that all series $\alpha
_{p}\left( y\right) $ converge and give by the equation $\left( 7\right) $,\
the zeros of $B_{y}\left( f^{\flat }\right) $ for $\left| y\right| <1$. We
then show as above that all zeros of $B_{y}\left( f^{\flat }\right) $ are
real for $y\in \left[ 0,1\right] $, and, in particular, that RH is true.
Hence, we just prove the implication : 
\begin{equation}
R^{\ast }\geq 1\text{ for the function }f^{\flat }\Longrightarrow RH\text{
is true}
\end{equation}
Note that to prove the condition $R^{\ast }\geq 1$, we cannot use Theorem 16
of $\left[ 3\right] $ for which the lower bound $R^{\ast }\geq \rho _{o}^{2}$
is optimum. So we have to seek a majoration of $v_{p}\left( q\right) $
different from that used in this theorem, and probably specific to the
function $f^{\flat }$. This raises difficult technical problems. So we
cannot, at present, increase the convergence radius of the power series $%
\alpha _{p}\left( y\right) $. That's why we have limited ourselves in the
rest of the article, to study the reality of the zeros of $B_{y}\left(
f^{\flat }\right) $ when $y$ increases along the real segment $\left[ \rho
_{o}^{2},1\right] $, in renouncing a priori to the convergence of the power
series $\alpha _{p}\left( y\right) $ in the disk $\left| y\right| <1$.

\section{Background on the Laguerre entire functions}

\QTP{Body Math}
\bigskip For the reader's convenience, we recall in this section the
definitions and some known properties of these functions (see for example $%
\left[ 12\right] $ or $\left[ 7\right] $).

\begin{definition}
$i)$ We say that an entire function $f$\ is of the first type of Laguerre
(abbreviated $L_{1}$) if it can be written 
\begin{equation*}
f\left( x\right) =Cx^{m}e^{\sigma x}\prod\limits_{p\geq 1}\left(
1+a_{p}x\right) 
\end{equation*}
with $C,\sigma ,a_{p}$ $\geq 0$, $m$ $\in \mathbb{N}$, and $\sum_{p\geq
1}a_{p}<+\infty $

\bigskip $ii)$ We say that an entire function $f$ is of the second type of
Laguerre (abbreviated $L_{2}$) if it can be written 
\begin{equation*}
f\left( x\right) =Cx^{m}e^{-ax^{2}+bx}\prod\limits_{p\geq 1}\left[ \left(
1+a_{p}x\right) e^{-a_{p}x}\right] 
\end{equation*}
with $C,b,a_{p}\in \mathbb{R}$, $a\geq 0$, $m\in \mathbb{N}$, and $%
\sum_{p\geq 1}a_{p}^{2}<+\infty $
\end{definition}

\bigskip Note that in $\left[ 7\right] $, the Laguerre entire functions are
called the Laguerre-Polya class.

Let us give two others definition that will be useful

\begin{definition}
\bigskip Let $f\left( x\right) =\sum_{n\geq 0}A_{n}x^{n}$be a formal power
series on $\mathbb{C}$, finite or infinite. For $n\in \mathbb{N}$, the $n$%
-th Jensen polynomial of $f$ is the polynomial (we follow the definition
given in $\left[ 7\right] $ Section 3) 
\begin{equation}
J_{n}\left( x\right) =\sum\limits_{k=0}^{n}\dfrac{n!}{\left( n-k\right) !}%
A_{k}x^{k}\text{ .}
\end{equation}
\end{definition}

\begin{definition}
We say that a formal power series $f\left( x\right) =$ $\sum_{n\geq
0}A_{n}x^{n}$ is without gaps, if it can be written 
\begin{equation*}
f\left( x\right) =\sum\limits_{n=n_{1}}^{n_{2}}A_{n}x^{n}\text{ with }%
A_{n}\neq 0\text{ \ \ for }0\leq n_{1}\leq n\leq n_{2}\leq +\infty \text{ .}
\end{equation*}
For any formal power series without gaps, we put $\Omega
_{n}=A_{n-1}A_{n+1}/A_{n}^{2}$ for $n_{1}\leq n\leq n_{2}$ and \ $\Omega
_{n}=0$ \ otherwise.
\end{definition}

\bigskip Here are the known properties of Laguerre entire functions we need:

\begin{proposition}
-If $f$ is $L_{1}$, then it is $L_{2}$.

\bigskip - If $f$\ is $L_{2}$ and if all the coefficients of its MacLaurin
series are positive, then it is $L_{1}$.

- If $f\left( x\right) =\sum_{n\geq 0}A_{n}x^{n}$ is $L_{1}$, then it is
without gaps and we have the inequality (of Turan) 
\begin{equation}
\text{\ \ \ \ }\Omega _{n}\leq \dfrac{n}{n+1}\text{ for all }n\geq 1\text{ .}
\end{equation}

- \bigskip The set of functions $L_{1}$ (respectively $L_{2}$) is closed for
the topology of uniform convergence on compact subsets of $\mathbb{C}$, and
stable by derivation.
\end{proposition}

\begin{theorem}
\bigskip (Laguerre) let $f\left( x\right) =\sum_{n\geq 0}A_{n}x^{n}$ be $%
L_{2}$ and let $\varphi $ be $L_{2}$ without positive zeros, then the
function $\sum_{n\geq 0}A_{n}\varphi \left( n\right) x^{n}$ is $L_{2}$.
\end{theorem}

\begin{theorem}
\bigskip Let $f\left( x\right) =\sum_{n\geq 0}A_{n}x^{n}$ be an entire
function with coefficient in $\mathbb{C}$, such that $f\left( 0\right) \neq 0
$, then (Jensen)

$-$ $f$ is $L_{1}$ if and only if its Jensen polynomials have only negative
\ zeros.

$-$ $f$ is $L_{2}$ if and only if its Jensen polynomials have only real\
zeros.

Moreover (Polya, $\left[ 14\right] $) if $f$ is $L_{2}$ and not of the form $%
P\left( x\right) e^{\alpha x}$, where $P$ is a polynomial and $\alpha $ a
real number different from zero, its Jensen polynomials have only \textbf{%
simple} real zeros.\bigskip 
\end{theorem}

\begin{remark}
\bigskip We chek with $\left( 4\right) $ and Definition 2, that RH is
equivalent to what $f^{\flat }$is $L_{1}$, as announced in Remark 1.
Conversely, if $f$ is $L_{1}$, Proposition 3.1 shows that $f$ has no gaps
and satisfies the inequality $\left( 17\right) $. Hence, we can apply to $f$
the method introduced in $\left[ 3\right] $ and summarized above. Futher,
the inequality $\left( 17\right) $ shows that $\Omega \leq 1$ for $f$,
therefore Theorem 2.2 applies exactly to $f$.
\end{remark}

\bigskip

\section{An analogue of the De Bruijn-Newman\textbf{\ constant}}

\bigskip

\subsection{\protect\bigskip Some additional properties of the $y$-Borel
transform}

We establish a principle of contraction, we need to prove the following
theorems in the article:

\begin{theorem}
\bigskip \bigskip Let $f\left( x\right) =\sum_{n\geq 0}A_{n}x^{n}$ be an
entire function with real coefficients $A_{n}$.

$i)$ If $f$ is $L_{2}$, then $B_{y}\left( f\right) $ is also $L_{2}$ for $%
y\in \left[ 0,1\right] $.

$ii)$ If there is $y_{0}$ $\in \left] 0,1\right[ $ such that $%
B_{y_{0}}\left( f\right) $ has only real zeros, then $B_{y}\left( f\right) $
is $L_{2}$ for $y\in \left[ 0,y_{0}\right] $.
\end{theorem}

\bigskip

In $\left[ 2\right] $ (see also $\left[ 16\right] $ Definition 4.1 )
B\'{e}zivin introduced the following definition:

\begin{definition}
Let $q\in \mathbb{C}^{\ast }$, we say that a formal power series $f\left(
x\right) =\sum_{n\geq 0}A_{n}x^{n}$\textbf{\ }$\in \mathbb{C}\left[ \left[ x%
\right] \right] $ is $q$-Gevrey type $1$, if the power series 
\begin{equation*}
B_{q^{-1}}\left( f\right) \left( xq\right) =\sum_{n\geq
0}A_{n}x^{n}q^{-n\left( n-1\right) /2}
\end{equation*}
has a convergence radius different from zero.
\end{definition}

\bigskip First, we prove

\begin{lemma}
Let $q>0$ be such that the formal power series $f\left( x\right)
=\sum_{n\geq 0}A_{n}x^{n}$ $\in \mathbb{C}\left[ \left[ x\right] \right] $
is $q$-Gevrey type 1. Then for $\left| y\right| <1/q$, \ $B_{y}\left(
f\right) $ is an entire function of order zero. In particular if $f$ is an
entire function, $B_{y}\left( f\right) $ is an entire function of order zero
for $\left| y\right| <1$.
\end{lemma}

\begin{proof}
It is easy to see that $f$ is $q$-Gevrey type $1$, is equivalent to the fact
that there exist $C,A>0$ such that for $n\in \mathbb{N}$%
\begin{equation}
\left| A_{n}\right| <Cq^{\tfrac{n\left( n+1\right) }{2}}A^{n}\text{.}
\end{equation}
Let $\chi $ be any positive number and $y$ a nonzero complex number such
that $\left| y\right| <1/q$. We have then for $n\in \mathbb{N}$ 
\begin{equation*}
n^{\chi }\left| A_{n}\,y^{n\left( n+1\right) /2}\right| ^{1/n}<n^{\chi
}C^{1/n}\left( \left| y\right| q\right) ^{\tfrac{n+1}{2}}A
\end{equation*}
with $\left| y\right| q<1.$ Hence the left-hand side of the above equation
tends to z\'{e}ro as $n\longrightarrow \infty $. There is therfore a naturel
integer $n_{1}$ such that for $n\geq n_{1}$ we have 
\begin{equation}
\left| A_{n}\,y^{n\left( n+1\right) /2}\right| ^{1/n}\leq \dfrac{1}{n^{\chi }%
}
\end{equation}
Let us write $B_{y}\left( f\right) =\sum_{n\geq 0}C_{n}x^{n}$, ie $%
C_{n}=A_{n}y^{n\left( n+1\right) /2}$. Taking $\chi =1$, equation $\left(
19\right) $ gives $\lim_{n\rightarrow \infty }\sqrt[n]{\left| C_{n}\right| }%
=0$, what already shows that $B_{y}\left( f\right) $ is an entire function.
This implies in particular that there is a naturel integer $n_{2}$ such that
for $n\geq n_{2}$ 
\begin{equation}
\log \left| C_{n}^{-1}\right| >0
\end{equation}
We set $n_{0}=\max \left( n_{1},n_{2}\right) $. From $\left( 19\right) $ and 
$\left( 20\right) $ we get for $n\geq n_{0}$ 
\begin{equation*}
0\leq \dfrac{n\log n}{\log \left| C_{n}^{-1}\right| }\leq \dfrac{1}{\chi }.
\end{equation*}
Since $\chi $ is arbitrary in $\mathbb{R}_{+}^{\ast }$, we deduce 
\begin{equation*}
\lim_{n\rightarrow \infty }\dfrac{n\log n}{\log \left| C_{n}^{-1}\right| }=0
\end{equation*}
It follows from Theorem 14.1.1 of $\left[ 11\right] $ that the order of $%
B_{y}\left( f\right) $ is zero.

Now if $f$ is an entire function, then $\lim_{n\rightarrow \infty }\sqrt[n]{%
\left| A_{n}\right| }=0$. There is thus a natural integer $n^{\prime }$ such
that, for $n\geq n^{\prime }$\ we have $\left| A_{n}\right| <1$. Hence, with 
$C=\max_{0\leq n\leq n^{\prime }}\left\{ \left| A_{n}\right| +1\right\} $
and $A=1$, $\left( 18\right) $ shows that $f$ is $1$-Gevrey type 1. $%
B_{y}\left( f\right) $ is therefore an entire function of order zero for $%
\left| y\right| <1$.
\end{proof}

\bigskip

\begin{proof}[Proof of Theorem 4.1]
$i)$ If $f$ is $L_{2}$, and $y\in \left] 0,1\right] $. We consider the
function $\varphi :\mathbb{C}$ $\longrightarrow \ \mathbb{C}$ defined by 
\begin{equation*}
\varphi \left( x\right) =y^{\tfrac{x\left( x+1\right) }{2}}=\exp \left( 
\frac{\ln y}{2}x^{2}+\frac{\ln y}{2}x\right) 
\end{equation*}
$\varphi $ is an entire function and $L_{2}$ because $\ln y\leq 0$. And it
has no positive zeros (because it has no zeros). The application of Theorem
3.2\ to $\ f$ and $\varphi $, shows that the function 
\begin{equation*}
\sum\limits_{n\geq 0}A_{n}\varphi \left( n\right) x^{n}=B_{y}\left( f\right)
\left( x\right) 
\end{equation*}
is $L_{2}$. The case $y=0$ is obvious since $B_{0}\left( f\right) \left(
x\right) =A_{0}$.

$ii)$ If $f$ is an entire function, and $y_{0}\in \left] 0,1\right[ $, we
know from Lemma 4.2 that $B_{y_{0}}\left( f\right) $ is an entire function
of order zero. Thus, it can be written 
\begin{equation*}
B_{y_{0}}\left( f\right) \left( x\right) =A_{m}y_{0}^{m\left( m+1\right)
/2}x^{m}\prod\limits_{p\geq 1}\left( 1+a_{p}x\right) \text{ ,}
\end{equation*}
where $A_{m}$ is the first non-zero coefficient of the Maclaurin series of $f
$,\ and where the product is normally convergent on each compact subset of $%
\mathbb{C}$. It follows that $B_{y_{0}}\left( f\right) $ is $L_{2}$ if $f$\
is real and $B_{y_{0}}\left( f\right) $ has only real zeros.

For $y$ $\in \left[ 0,y_{0}\right] $, we get by $\left( 6\right) $%
\begin{equation*}
\text{\ }B_{y}\left( f\right) =B_{\tfrac{y}{y_{0}}}\left( B_{y_{0}}\left(
f\right) \right) 
\end{equation*}
with $\dfrac{y}{y_{0}}\in \left[ 0,1\right] $.Thus, it follows from the
previous cases that\ $B_{y}\left( f\right) $ is $L_{2}$.
\end{proof}

\bigskip

\subsection{\protect\bigskip Constant $\Theta $\ associated with $%
B_{y}\left( f^{\flat }\right) $}

\bigskip

\begin{definition}
\bigskip Let $f^{\flat }$ be given by $\left( 3\right) $. We define 
\begin{equation*}
\Theta =\sup \left\{ y\in \left[ 0,+\infty \right[ \text{ such that }%
B_{y}\left( f^{\flat }\right) \text{ is an entire function and has only real
zeros}\right\} 
\end{equation*}
\end{definition}

\bigskip We are going to show that $\Theta $ satisfies the following
equivalences :

\begin{equation}
\left\{ 
\begin{tabular}{l}
{{*} For }$y\geq 0$ we have: $y\leq \Theta \text{ }\Longleftrightarrow \text{%
\ }\left\{ B_{y}\left( f^{\flat }\right) \text{ is an entire function whose
zeros are all real}\right\} $ \\ 
{*} $1\leq \Theta $ \ \ $\Longleftrightarrow $ \ RH
\end{tabular}
\right\}
\end{equation}
In fact, we will prove the more accurate following result:

\begin{theorem}
\bigskip We have

$i)$ If $0\leq y\leq \Theta $, $B_{y}\left( f^{\flat }\right) $ is $L_{1}$
and has thus only negative zeros.

$ii)$ If $\Theta <1$ and $\Theta <y\leq 1$ $\ B_{y}\left( f^{\flat }\right) $
is an entire function with at least two non-real zeros.

$iii)$ If \ $1<y$ \ \ the convergence radius of $B_{y}\left( f^{\flat
}\right) $ is zero.

$iv)$ $\Theta =1$ if and only if the Riemann hypothesis is true.

$v)$ We have $\ \ 0<\rho _{o}^{2}<1/4\leq \Theta \leq 1$.
\end{theorem}

\bigskip This theorem is one of the main results of the article. In
particular, the point $iv)$ gives a reformulation of RH. To prove this
theorem, we need two Lemmas.

\begin{lemma}
Let $f\left( x\right) =\sum_{n\geq 0}A_{n}x^{n}$ be an entire function with $%
A_{n}\neq 0$ for all $n$ and satisfying $\lim \inf_{n\geq 1}\left| \Omega
_{n}\right| \geq 1$. Then the radius of convergence of $B_{y}\left( f\right) 
$ is zero for all $y\in \mathbb{C}$ such that $\left| y\right| >1$.
\end{lemma}

\begin{proof}
\bigskip Let $y\in \mathbb{C}$, $\left| y\right| >1$ we have $\left|
y\right| >\mu =\dfrac{\left| y\right| +1}{2}>1.$ By assumption, there exists
an integer $n_{0}\geq 1$ such that for $n\geq n_{0}$ we have $\left| \Omega
_{n}\right| >\mu ^{-1}$. In addition, we have for all $n\geq 2$: 
\begin{equation*}
\dfrac{A_{n}}{A_{n-1}}=\Omega _{n-1}\Omega _{n-2}...\Omega _{1}\dfrac{A_{1}}{%
A_{0}}.
\end{equation*}
If, as above $C_{n}=A_{n}y^{n\left( n+1\right) /2}$, we have for $n>n_{0}$ 
\begin{equation*}
\dfrac{C_{n}}{C_{n-1}}=\Omega _{n-1}...\Omega _{n_{0}}y^{n-n_{o}}M\text{ \ \
with }M=\Omega _{n_{0}-1}...\Omega _{1}y^{n_{0}}\dfrac{A_{1}}{A_{0}}
\end{equation*}
thus 
\begin{equation*}
\left| \dfrac{C_{n}}{C_{n-1}}\right| \geq \left( \mu ^{-1}\left| y\right|
\right) ^{n-n_{0}}\left| M\right| .
\end{equation*}
$M$ is a constant different from zero and $\mu ^{-1}\left| y\right| >1$. It
follows that the convergence radius of $B_{y}\left( f\right) $ is zero.
\end{proof}

\bigskip

\begin{lemma}
\bigskip \bigskip \bigskip For any formal power without gaps $f\left(
x\right) =\sum_{n\geq 0}A_{n}x^{n}\in \mathbb{C}\left[ \left[ x\right] %
\right] $, satisfying $A_{0}\neq 0$ and $\Omega <+\infty $, the $y$-Borel
transform of $f$ is an entire function of order zero for $\left| y\right|
<1/\Omega $. In addition, if the coefficients $A_{n}$ are positive, $%
B_{y}\left( f\right) $ is a function $L_{1}$ for $y\in \left[ 0,1/4\Omega %
\right] .$
\end{lemma}

\begin{proof}
\bigskip With an easy recurrence, we have for $n\geq 1$%
\begin{equation*}
A_{n}=\Omega _{n-1}\left( \Omega _{n-2}\right) ^{2}...\left( \Omega
_{1}\right) ^{n-1}\dfrac{A_{1}^{n}}{A_{0}^{n-1}}
\end{equation*}
then 
\begin{equation*}
\left| A_{n}\right| \leq \Omega ^{\tfrac{n\left( n+1\right) }{2}}\left| 
\dfrac{A_{1}}{A_{0}\Omega }\right| ^{n}\left| A_{0}\right| 
\end{equation*}
By comparison with the inequality $\left( 18\right) $, it follows that $f$
is $q$-Gevrey type 1 for $q=\Omega $. Lemma 4.2 then shows that $B_{y}\left(
f\right) $ is an entire function of order zero when $\left| y\right|
<1/\Omega $.

Assume now $A_{n}>0$ and consider the Jensen polynomials of $f$, $J_{n}$. It
is easy to see that the Jensen polynomials of the $y$-Borel transform of $f$
\ are the $y$-Borel transform of the Jensen polynomials of $f$, ie 
\begin{equation}
B_{y}\left( J_{n}\right) \left( x\right) =\sum\limits_{k=0}^{n}\dfrac{n!}{%
\left( n-k\right) !}A_{k}x^{k}y^{\tfrac{k\left( k+1\right) }{2}%
}=\sum\limits_{k=0}^{n}C_{n,k}x^{k}\text{.}
\end{equation}
For $y>0$, we have $C_{n,k}>0$ and for $1\leq k\leq n-1$%
\begin{equation*}
\dfrac{C_{n,k-1}C_{n,k+1}}{C_{n,k}^{2}}=\dfrac{A_{k-1}A_{k+1}}{A_{k}^{2}}%
\dfrac{\left( n-k\right) }{\left( n-k+1\right) }y
\end{equation*}
and then, for $y\in \left[ 0,1/4\Omega \right] $ 
\begin{equation*}
\dfrac{C_{n,k-1}C_{n,k+1}}{C_{n,k}^{2}}<1/4
\end{equation*}
It follows from a theorem of Kurtz ( see section 4 of $\left[ 7\right] $),
\bigskip that all the zeros of $B_{y}\left( J_{n}\right) $ are negative.
This being true for any $n\in \mathbb{N}$, we deduce from \ the theorem 3.3
that $B_{y}\left( f\right) $ is $L_{1}$for $y\in \left] 0,1/4\Omega \right] $%
. The case $y=0$ is obvious.
\end{proof}

\bigskip

\begin{proof}[Proof of Theorem 4.3]
\bigskip According to the double inequality $\left( 14\right) $, we have $%
\lim_{n\rightarrow \infty }\Omega _{n}^{\flat }=1$. Thus Lemma 4.4 shows
that for $y>1$, $B_{y}\left( f^{\flat }\right) $ has a convergence radius
equal to zero, which proves $iii)$. This also shows that $\Theta \leq 1$, ie
the last inequality of $v)$.

Now let $y$ be such that $0\leq y<\Theta \leq 1$. According to the
definition of $\Theta $, there is $y_{0}$ such that $y<y_{0}<\Theta \leq 1$
and such that $B_{y_{0}}\left( f^{\flat }\right) $ has only real zeros. It
follows from item $ii)$ of\ Theorem 4.1 that $B_{y}\left( f^{\flat }\right) $
is $L_{2}$. In addition, it is clear that the coefficients of the Maclaurin
series of $B_{y}\left( f^{\flat }\right) $\ are positive. Therefore, $%
B_{y}\left( f^{\flat }\right) $ is actually $L_{1}$ with Proposition 3.1.
Since $L_{1}$ is closed for the topology of uniform convergence on each
compact subset of $\mathbb{C}$ (see Proposition 3.1), we deduce from
Proposition 2.1 that $B_{\Theta }\left( f^{\flat }\right) $ is also $L_{1}$.
This proves $i)$.

We reason by contradiction to prove $ii)$. Assume that $\Theta <1$ and that
there is $y_{0}\in \left] \Theta ,1\right] $ such that $B_{y_{0}}\left(
f^{\flat }\right) $ has only real zeros.

- If $y_{0}\in \left] \Theta ,1\right[ $ the point $ii)$ of\ Theorem 4.1
shows that for $y\in \left[ 0,y_{0}\right] $ the zeros of $B_{y}\left(
f^{\flat }\right) $ are real numbers, in contradiction with the definition
of $\Theta $.

- If $y_{0}=1$ the function $B_{1}\left( f^{\flat }\right) =f^{\flat }$ is $%
L_{2}$ because we know that $f^{\flat }$ is written as $f^{\flat }\left(
x\right) $=$A_{0}^{\flat }\prod_{p\geq 1}\left( 1-\dfrac{x}{x_{p}}\right) $
with $A_{0}^{\flat }>0$. The point $i)$ of\ Theorem 4.1, then shows that the
zeros of $B_{y}\left( f^{\flat }\right) $ are real for $y\in \left[ 0,1%
\right] $, in contradiction with $\Theta <1$.

For the point $iv)$, the implication: RH $\Longrightarrow \Theta =1$ is now
clear with $ii)$, since we already know that $\Theta \leq 1$. The reciprocal
implication follows from $i)$.

To complete the proof of $v)$ we apply Lemma 4.5 to $f^{\flat }$ with $%
\Omega =1$. Hence $B_{y}\left( f^{\flat }\right) $ is $L_{1}$ and has all
its zeros real for $y\in \left[ 0,1/4\right] $, which proves $\Theta \geq 1/4
$
\end{proof}

\begin{remark}
\bigskip We can now prove that for the function $f^{\flat }$,$\ $we have $%
\rho _{o}^{2}\leq R^{\ast }\leq 1$. The inequality $\rho _{o}^{2}\leq
R^{\ast }$\ follows from Theorem 2.2. In addition, assume that $R^{\ast }>1$%
, then the theorem 3 of $\left[ 2\right] $ shows that we have for $\left|
y\right| <R^{\ast }$%
\begin{equation*}
B_{y}\left( f^{\flat }\right) \left( x\right) =A_{0}^{\flat
}\prod\limits_{p\geq 1}\left( 1+\alpha _{p}\left( y\right) x\right) \text{ ,}
\end{equation*}
where the product is normally convergent on any compact subset of $\mathbb{C}
$. Thus, $B_{y}\left( f^{\flat }\right) $ would be an entire function for $%
y\in \left] 1,R^{\ast }\right[ $ contradicting $iii)$ of Theorem 4.3.
\end{remark}

\subsection{\protect\bigskip Analogy of $\Theta $ with the de Bruijn-Newman
constant}

In $\left[ 13\right] $ Newman, working after de Bruijn $\left[ 6\right] $\
on the Fourier transforms with real zeros (see Remark 1), defines for $%
\lambda \in \mathbb{R}$ 
\begin{equation*}
\Xi _{\lambda }\left( z\right) =2\int\limits_{0}^{\infty }\exp \left(
\lambda u^{2}\right) \Phi \left( u\right) \cos zu\,du\text{ .}
\end{equation*}
That is the Fourier transform of $\exp \left( \lambda u^{2}\right) \Phi
\left( u\right) $ where 
\begin{equation*}
\Phi \left( u\right) =\sum\limits_{n=1}^{\infty }\left( 4\pi
^{2}n^{4}e^{9u/2}-6\pi n^{2}e^{5u/2}\right) \exp \left( -\pi
n^{2}e^{2u}\right) \text{ .}
\end{equation*}
We have $\Xi _{0}\left( z\right) =\Xi \left( z\right) $ where $\Xi $ is the
Riemann Xi function. The constant $\Lambda $ of de Bruijn-Newman, whose the
existence was proved in $\left[ 13\right] $, satisfies the properties (we
follow the presentation given in $\left[ 1\right] $ Section 3.2.2) : 
\begin{equation*}
\left\{ 
\begin{tabular}{l}
{*} For $\lambda \in \mathbb{R}$, we have: $\Lambda \leq 4\lambda
\Longleftrightarrow \left\{ \Xi _{\lambda }\text{ is an entire function
whose zeros are all real}\right\} $ \\ 
{*} $\Lambda \leq 0$ $\Longleftrightarrow $ RH
\end{tabular}
\right\}
\end{equation*}
By comparison with the properties $\left( 21\right) $, we see that the
constant $\Theta $ plays a role similar to that of the constant $\Lambda $.
It would be interesting to find an equation relating $\Theta $ to $\Lambda $%
. We leave aside the issue in the following of the article.

\section{Another reformulation of the Riemann hypothesis}

Here is the\bigskip\ reformulation:

\begin{theorem}
Let $A_{n}^{\flat }$ be defined by the equation $\left( 2\right) $\ where $%
\Xi $ is the Riemann Xi function given by $\left( 1\right) $. For $n\geq 0$,
denote by $\Delta _{n}\left( y\right) $ the discriminant of the following
polynomial 
\begin{equation*}
B_{y}\left( J_{n}^{\flat }\right) \left( x\right) =\sum_{k=0}^{n}\dfrac{n!}{%
\left( n-k\right) !}A_{k}^{\flat }x^{k}y^{k\left( k+1\right) 2}\ \text{.}
\end{equation*}
Then the Riemann hypothesis is true if and only if for all $n\geq 2$, $%
\Delta _{n}\left( y\right) $ does not vanish on the interval $\left[ 1/4,1%
\right[ $\bigskip .
\end{theorem}

\bigskip To prove this theorem, we need two lemmas.

\begin{lemma}
\bigskip Let $\mathbb{K}$ be a field of characteristic zero and $P\left(
x\right) =A_{0}+A_{1}x+...+A_{n}x^{n}\in \mathbb{K}\left[ x\right] $ with $%
A_{k}\neq 0$ for $0\leq k\leq n$. Then

i) The discriminant $\Delta \left( y\right) $ of $B_{y}\left( P\right) $\ is
a polynomial in $y$ not identically zero.

ii) If $\mathbb{K}=\mathbb{R}$, $\Omega =\sup_{1\leq k\leq n-1}\left|
A_{k-1}A_{k+1}/A_{k}^{2}\right| $ and if we assume the roots of $B_{y}\left(
P\right) $ are simple for $y\in \left[ \rho _{o}^{2}\Omega ^{-1},1\right[ $, 
$P$ has only real roots.
\end{lemma}

\begin{proof}
\bigskip $i)$ By application of Theorem 1 of $\left[ 3\right] $, we have in
the ring $\mathbb{K}\left[ \left[ x,y\right] \right] $: 
\begin{equation*}
B_{y}\left( P\right) \left( x\right) =A_{0}\prod\limits_{p=1}^{n}\left(
1+\alpha _{p}\left( y\right) x\right) 
\end{equation*}
with (see equations $\left( 8\right) $ and $\left( 9\right) $ above) 
\begin{equation}
\alpha _{p}\left( y\right) =\dfrac{A_{p}}{A_{p-1}}y^{p}+\sum%
\limits_{q=p+1}^{\infty }u_{p}\left( q\right) y^{q}\text{.}
\end{equation}
Thus the zeros of $B_{y}\left( P\right) $ considered as a polynomial in $x$
on the field $\mathbb{K}\left( \left( y\right) \right) $, can be written as
(see equation $\left( 7\right) $) 
\begin{equation}
x_{p}\left( y\right) =-\left( \alpha _{p}\left( y\right) \right) ^{-1}=-%
\dfrac{A_{p-1}}{A_{p}y^{p}}\left( 1+\sum\limits_{q=1}^{\infty }w_{p}\left(
q\right) y^{q}\right) 
\end{equation}
where $w_{p}\left( q\right) \in \mathbb{K}$. It follows that the polynomial $%
B_{y}\left( P\right) $ is separate over $\mathbb{K}\left( \left( y\right)
\right) $ and that its zeros are all distinct because they have different
valuations. This proves that the discriminant $\Delta \left( y\right) $ of $%
B_{y}\left( P\right) $ is a polynomial in $y$ not identically zero.

$ii)$ First, we assume $\rho _{o}^{2}\Omega ^{-1}<1$, otherwise Corollary 18
of $\left[ 3\right] $ shows that $P$ has only real roots, without
requirement of simplicity of the roots. We know by the theorem 16 of $\left[
3\right] $, that the series $\alpha _{p}\left( y\right) $ are absolutely
convergent in the disk $D\left( 0,\rho _{o}^{2}\Omega ^{-1}\right)
=\;\left\{ y\in \mathbb{C}\text{ : }\left| y\right| <\rho _{o}^{2}\Omega
^{-1}.\right\} $

Thus, they define holomorphic functions in this disk. Since $A_{n}y^{n\left(
n+1\right) /2}$ does not vanish outside $0$, it is the same for the
functions $\alpha _{p}\left( y\right) $. Therefore the $n$ zeros $%
x_{p}\left( y\right) =$ $-1/\alpha _{p}\left( y\right) $ of $B_{y}\left(
P\right) $ are also holomorphic functions in the punctured disk $D^{\ast
}\left( 0,\rho _{o}^{2}\Omega ^{-1}\right) =\left\{ y\in \mathbb{C}\text{ : }%
0<\left| y\right| <\rho _{o}^{2}\Omega ^{-1}.\right\} $. And their Laurent
series at the point $0$, given by $\left( 24\right) $, converge on this
punctured disk. It is well known that the roots of $Q\left( x,y\right)
=B_{y}\left( P\right) \left( x\right) =0$ are the branches of one or several
algebraic functions, that would be obtained by decomposing $Q$ into
irreductible factors in the factorial ring $\mathbb{C}\left[ x,y\right] .$
The following decomposition shows that the case of several algebraic
functions cannot be excluded. 
\begin{equation*}
Q\left( x,y\right) =B_{y}\left( 1+x+x^{2}+x^{3}\right)
=1+xy+x^{2}y^{3}+x^{3}y^{6}=\left[ 1+xy^{2}\right] \left[ 1+x\left(
y-y^{2}\right) +x^{2}y^{4}\right] 
\end{equation*}

However, we will follow the reasoning of Section 12.1 of $\left[ 11\right] $%
, which deals with the case of a single algebraic function. Indeed the
reasoning of Section 12.1 of $\left[ 11\right] $\ does not depend, for the
part that concern us, to the fact that $Q$ is irreductible or not.

For a fixed value $y\in \mathbb{C}\backslash S$ ($S$ is a set defined
below), the equation 
\begin{equation}
Q\left( x,y\right) =\sum_{k=0}^{n}A_{k}x^{k}y^{k\left( k+1\right) 2}=0
\end{equation}
has $n$ distinct finite roots. The set $S$\ of exceptional values is
composed of three subsets. The first of these is the subset of roots of $%
A_{n}y^{n\left( n+1\right) /2}=0$, ie $\left\{ 0\right\} $ ( the Laurent
series $\left( 24\right) $ shows that $y=0$ is actually a pole of order $p$
for $x_{p}\left( y\right) $). Secondly, we must exclude the values of $y$
for which $\left( 25\right) $ has multiple roots. It is well known that
these values of $y$ are the roots of the discriminant $\Delta \left(
y\right) $. We know from the point $i)$ of Lemma 5.2, that $\Delta \left(
y\right) $ is not identically nul, so this second subset is finite. Third,
we must exclude the point at infinity. $S$ is thus finite and the hypothesis
of the point $ii)$ of Lemma 5.2 shows that $S\cap \left[ \rho _{o}^{2}\Omega
^{-1},1\right[ =\emptyset $. Let $d$ be the distance between $S\backslash 
\mathbb{R}$ and $\mathbb{R}$ (we set $d=+\infty $ if $\ S\backslash \mathbb{R%
}=+\infty $). Put $b=\max \left\{ y\in S\cap \mathbb{R}\text{ : }y<\rho
_{o}^{2}\Omega ^{-1}\right\} $; we have $b\geq 0$ since $0\in S$. And
consider the real number $a=\left( b+\rho _{o}^{2}\Omega ^{-1}\right) /2$,
we have $a\in D^{\ast }\left( 0,\rho _{o}^{2}\Omega ^{-1}\right) $. Put $\mu
=1/2\min \left\{ d,a-b\right\} $, we have $\mu >0$. For each $p\in \left\{
1,2,...,n\right\} $, the function $x_{p}\left( y\right) $ holomorphic in $%
D^{\ast }\left( 0,\rho _{o}^{2}\Omega ^{-1}\right) $, can be expanded in a
Taylor series around the point $a$, which converges\ in a disk $\left|
y-a\right| <\nu _{p}$ ($\nu _{p}>0)$ . Put $\varepsilon =\min \left\{ \mu
,\nu _{1},\nu _{2},...,\nu _{n}\right\} $, we have $\varepsilon >0$.

For $y_{0}\in $ $\left[ \rho _{o}^{2}\Omega ^{-1},1\right] $, $\overline{T}%
\left( y_{0}\right) $ denotes the closed rectangular subset of the complex
plane, whose vertices are the complex numbers $a-\varepsilon +i\varepsilon $%
, $a-\varepsilon +i\varepsilon $, $y_{0}+i\varepsilon $, $y_{0}-i\varepsilon 
$. $T\left( y_{0}\right) $ denotes the interior of $\overline{T}\left(
y_{0}\right) $. Let $y_{0}$ be a fixed positive number in $\left] \rho
_{o}^{2}\Omega ^{-1},1\right[ $, the distance between $\overline{T}\left(
y_{0}\right) $ and $S$ is a positive number, $\delta \left( y_{0}\right) >0$%
. For $p\in \left\{ 1,2,...,n\right\} $, the Taylor series of $x_{p}\left(
y\right) $ around the point $a$ is convergent in the disk $\left| y-a\right|
<\varepsilon $, which is included in $T\left( y_{0}\right) $ for all $%
y_{0}\in \left] \rho _{o}^{2}\Omega ^{-1},1\right[ $. Following the
reasoning of section 12.1 of $\left[ 11\right] $ for the open subset $%
T\left( y_{0}\right) $, the implicit function theorem shows that each of
these Taylor series can be continued analytically along every path in $%
T\left( y_{0}\right) $. Since $T\left( y_{0}\right) $ is a simply-connected
domain, the monodromy theorem (see $\left[ 11\right] $ p.12) shows that
these analytic continuations provide for each $p\in \left\{
1,2,...,n\right\} $, a holomorphic function in $T\left( y_{0}\right) $,
which coincides respectively with $x_{p}\left( y\right) $ in the disk $%
\left| y-a\right| <\varepsilon .$ It follows from the principle of analytic
continuation, that we have for each $p\in \left\{ 1,2,...,n\right\} $ an
analytic continuation of $x_{p}\left( y\right) $, which is now defined and
holomorphic in $D^{\ast }\left( 0,\rho _{o}^{2}\Omega ^{-1}\right) \cup
T\left( y_{0}\right) .$

If we do this for all $y_{0}\in \left] \rho _{o}^{2}\Omega ^{-1},1\right[ $,
we get for each $p\in \left\{ 1,2,...,n\right\} $ an analytic continuation
of $x_{p}\left( y\right) $, that we denote $\widehat{x_{p}}\left( y\right) $%
, holomorphic in the open subset 
\begin{equation*}
O=D^{\ast }\left( 0,\rho _{o}^{2}\Omega ^{-1}\right) \cup T\left( 1\right)
=\bigcup_{y_{0}\in \left] \rho _{o}^{2}\Omega ^{-1},1\right[ }D^{\ast
}\left( 0,\rho _{o}^{2}\Omega ^{-1}\right) \cup T\left( y_{0}\right) \text{.}
\end{equation*}
Note that $\left] 0,1\right[ \subset $ $O$. By construction, for $p\in
\left\{ 1,2,...,n\right\} $ and $y\in O$, $\widehat{x_{p}}\left( y\right) $
is a zero of $B_{y}\left( P\right) $. In addition, since $Q\left( 0,y\right)
=A_{0}\neq 0$ we have $\widehat{x_{p}}\left( y\right) \neq 0$ in $O$ for
each $p\in \left\{ 1,2,...,n\right\} $. Hence, the law of permanence of
functional equations (see section 10.7 of $\left[ 11\right] )$ shows that
the relation 
\begin{equation*}
Q\left( x,y\right) =A_{0}\prod\limits_{p=1}^{n}\left( 1-\dfrac{x}{\widehat{%
x_{p}}\left( y\right) }\right) =0
\end{equation*}
is still valid for $x\in \mathbb{C}$ and $y\in $ $O$. It follows that the $n$
values $\widehat{x_{p}}\left( y\right) $ are \textbf{exactly} the roots of $%
B_{y}\left( P\right) $ for $y\in O$.

It is clear that the domain $O$ is symmetrical about the real axis. By
assumption, the coefficients $A_{k}$ are real, thus $\alpha _{p}\left(
y\right) $ and $x_{p}\left( y\right) =\left( \alpha _{p}\left( y\right)
\right) ^{-1}$ are also real on the nonempty open interval $\left] 0,\rho
_{o}^{2}\Omega ^{-1}\right[ $, where they coincide with, respectively, $%
\widehat{\alpha _{p}}\left( y\right) $ and $\widehat{x_{p}}\left( y\right) $%
. Hence, applying the Schwarz reflection principle to $O$ and each
holomorphic function $\widehat{x_{p}}\left( y\right) $, we find that $%
\widehat{x_{p}}\left( y\right) $ is a real number for\ $y\in \left] 0,1%
\right[ $ and $1\leq p\leq n$. In fact,we have just shown that the zeros of $%
B_{y}\left( P\right) $ are all real for $y\in \left] 0,1\right[ $ . The
proposition 2.1 then allows to conclude.
\end{proof}

\begin{lemma}
\bigskip\ Let $f\left( x\right) =\sum_{n\geq 0}A_{n}x^{n}$ be an entire
function and for $n\in \mathbb{N}$, let $J_{n}$ be the $n$-th Jensen
polynomial of $f$.

$i)$ If $\ $the coefficients $A_{n}$\ are real, different from zero and
satisfy the condition $\Omega <+\infty $, we have 
\begin{equation*}
f\text{ is }L_{2}\text{ (and therefore its zeros are real)}%
\Longleftrightarrow \left\{ 
\begin{array}{c}
\text{the polynomials }B_{y}\left( J_{n}\right) \text{ } \\ 
\text{have only simple zeros for }y\in \left[ \rho _{o}^{2}\Omega ^{-1},1%
\right[ 
\end{array}
\right\} 
\end{equation*}

$ii)$ If the coefficients $A_{n}$\ are strictly positive, we have 
\begin{equation*}
f\text{ is }L_{1}\text{ (and therefore its zeros are negative)}%
\Longleftrightarrow \left\{ 
\begin{array}{c}
\Omega <+\infty \text{ \ and the polynomials }B_{y}\left( J_{n}\right) \text{
} \\ 
\text{have only simple zeros for }y\in \left[ 1/4\Omega ,1\right[ 
\end{array}
\right\} 
\end{equation*}
\end{lemma}

\begin{proof}
\bigskip $i)$ Suppose first that for all $n\in \mathbb{N}$, $B_{y}\left(
J_{n}\right) $ has only simple zeros for $y\in \left[ \rho _{o}^{2}\Omega
^{-1},1\right[ $. As before, we assume $\rho _{o}^{2}\Omega ^{-1}<1$,
otherwise Corollary 18 of $\left[ 3\right] $ and Lemma 4.2 above, show that $%
B_{y}\left( f\right) $ is $L_{2}$ for $y\in \left[ 0,1\right[ $ ; and we
deduce from Proposition 2.1 and the last point of Proposition 3.1 that $f$
is $L_{2}$ without requirement of simplicity of zeros. We have 
\begin{equation*}
J_{n}\left( x\right) =\sum\limits_{k=0}^{n}C_{n,k}x^{k}\text{ with }%
C_{n},_{k}=\dfrac{n!}{\left( n-k\right) !}A_{k}\in \mathbb{R}^{\ast }
\end{equation*}
and for $1\leq k\leq n-1$%
\begin{equation*}
\left| \dfrac{C_{n,k-1}C_{n,k+1}}{C_{k}^{2}}\right| =\dfrac{n-k}{n-k+1}%
\left| \dfrac{A_{k-1}A_{k+1}}{A_{k}^{2}}\right| \leq \dfrac{n-1}{n}\left|
\Omega _{k}\right| 
\end{equation*}
thus 
\begin{equation*}
\sup_{1\leq k\leq n-1}\left| \dfrac{C_{n,k-1}C_{n,k+1}}{C_{n,k}^{2}}\right|
\leq \Omega =\sup_{k\geq 1}\left| \Omega _{k}\right| 
\end{equation*}
We then apply Lemma 5.2 to $J_{n}$, which shows that this polynomial has
only real zeros. Since this is true for any $n\in \mathbb{N}$, it follows
from Theorem 3.3 that $f$ is $L_{2}$.

Conversely, if we assume that $f$ is $L_{2}$, we know from Theorem 4.1\ that 
$B_{y}\left( f\right) $ is $L_{2}$ for all $y\in \left[ 0,1\right[ $.
Furthermore, Lemma 4.2 shows that the order of $B_{y}\left( f\right) $ is
zero in this interval. Thus $B_{y}\left( f\right) $ is not of the form $%
P\left( x\right) e^{\alpha x}$ where $P$ is a polynomial and $\alpha $ a
nonzero real number. It follows from the last point of Theorem 3.3 that the
Jensen polynomials of $B_{y}\left( f\right) $ have only simple zeros for $%
y\in \left[ 0,1\right[ $ and in particular for $\left[ \rho _{o}^{2}\Omega
^{-1},1\right[ $.

$ii)$ Assume first that $\Omega <+\infty $ ( we assume $1/4\Omega <1$,
otherwise Lemma 4.5 shows that $f$ is $L_{1}$ without requirement of
simplicity of zeros) and that for all $n\in \mathbb{N}$ the polynomial $%
B_{y}\left( J_{n}\right) $ has only simple zeros for $y\in \left[ 1/4\Omega
,1\right[ $. We know by Lemma 4.5 that $B_{y}\left( f\right) $\ \ is $L_{1}$%
, thus also $L_{2}$ for $y\in $\ $\left[ 0,1/4\Omega \right] $ and moreover
it is zero order. It is thus not of the form $P\left( x\right) e^{\alpha x}$
where $P$ is a polynomial and $\alpha $ a nonzero real number. And the zeros
of $B_{y}\left( J_{n}\right) $ are then also simple for $y\in \left[
0,1/4\Omega \right] $ by Theorem 3.3. Hence, the zeros of $B_{y}\left(
f\right) $ are simple for $y\in \left[ 0,1\right[ $ and we can apply the
results of $i)$, which shows that $f\ $is $L_{2}$. It follows from
Proposition 3.1\ that $f$ is also $L_{1}$. We deduce the converse from $i)$
and the fact that $\Omega \leq 1<+\infty $ by inequality $\left( 17\right) $.
\end{proof}

\begin{proof}[Proof of Theorem 5.1]
let $f^{\flat }$\bigskip\ be the function defined by $\left( 3\right) $. We
already now that The Riemann hypothesis is equivalent to that $f^{\flat }$
is $L_{1}$. We notice now that the polynomials 
\begin{equation*}
B_{y}\left( J_{n}^{\flat }\right) \left( x\right) =\sum_{k=0}^{n}\dfrac{n!}{%
\left( n-k\right) !}A_{k}^{\flat }x^{k}y^{k\left( k+1\right) 2}
\end{equation*}
are the $y$-Borel transform of the Jensen polynomials of $f^{\flat }$. We
then use the $ii)$ of Lemma 5.3, taking into account the fact that we have $%
\Omega =1$ for $f^{\flat }$ (see Section 2.3). Hence RH is true if and only
if $B_{y}\left( f^{\flat }\right) $ has only simple zeros for $y\in \left[
1/4,1\right[ $ and $n\in \mathbb{N}$. This is obvious for $n=0$ and $n=1$,
and for $n\geq 2$ this amounts to saying that $\Delta _{n}\left( y\right) $
does not vanish in this interval.
\end{proof}

\begin{remark}
\bigskip It is well known that $\Delta _{n}$ $\left( y\right) $ can be
expressed as a determinant. Theorem 5.1 can then be compared to a condition
equivalent to RH, given in Section C.8 of $\left[ 17\right] $, and using a
different sequence of determinants.
\end{remark}

\bigskip\ 

\section{\protect\bigskip Simplicity of the zeros of $B_{y}\left( f^{\flat
}\right) $}

A new necessary condition for the RH, follows from the following general
theorem.

\begin{theorem}
If $f$ is $L_{2}$ with $f\left( 0\right) \neq 0$, then the zeros of $%
B_{y}\left( f\right) $ are simple (and real) for $y\in \left] 0,1\right[ $.
\end{theorem}

\bigskip

\begin{corollary}
\bigskip \bigskip We can improve Theorem 4.3 by replacing the point $i)$ of
this theorem by:

$i^{\ast })$ If $0\leq y\leq \Theta $, $B_{y}\left( f^{\flat }\right) $ is $%
L_{1}$ and thus, has only negative zeros. Furthermore, the zeros are simple
if $0\leq y<\Theta $.

In particular, the Riemann hypothesis implies that the zeros of $B_{y}\left(
f^{\flat }\right) $ are simple for $y\in \left] 0,1\right[ $.
\end{corollary}

\begin{proof}
The first part of the point $i^{\ast })$ is the point $i)$ of Theorem 4.3,
we have already proved. It then suffices to apply Theorem 6.1 to $B_{\Theta
}\left( f^{\flat }\right) $, to deduce that the zeros of $B_{y}\left(
f^{\flat }\right) $ are simple for $0\leq y<\Theta $. Further, RH implies
that $\Theta =1$ (see $iv)$ of Theorem 4.3), and the final implication
follows from $i^{\ast })$.
\end{proof}

We need the following lemma to prove Theorem 6.1

\begin{lemma}
\bigskip Let $Q\left( x,y\right) =\sum_{n=0}^{N}A_{n}x^{n}y^{n\left(
n+1\right) /2}$ with $A_{n}\in \mathbb{C}$ and $N$ a natural number greater
than $1$ or $N=$ $+\infty $. If $N=+\infty $, assume that the power series
in two variables $x$ and $y$ converges in a neighborhood of $\left(
a,b\right) \in \mathbb{C}^{2}$ with $a$ and $b$ different from zero. Assume
further that 
\begin{equation*}
\text{ }Q\left( a,b\right) =0\text{, \ \ }\dfrac{\partial Q}{\partial x}%
\left( a,b\right) =0\text{, \ \ }\dfrac{\partial ^{2}Q}{\partial x^{2}}%
\left( a,b\right) \neq 0\text{.}
\end{equation*}
Then, there is a convergent Puiseux series, $\psi \left( h\right) $,
satisfying $Q\left( a+\psi \left( h\right) ,b+h\right) $\bigskip $=0$, and
whose the first term is 
\begin{equation}
\left( -\dfrac{a^{2}}{b}h\right) ^{1/2}
\end{equation}
\end{lemma}

\begin{proof}
\bigskip Let $T\left( h,k\right) $ be the Taylor expansion of $Q\left(
x,y\right) $ about $\left( a,b\right) $, taken as 
\begin{equation*}
T\left( h,k\right) =Q\left( a+k,b+h\right) =\sum_{n=0}^{N\left( N+3\right)
/2}\sum_{i+j=n}a_{i,j}h^{i}k^{j}\text{ with }a_{i,j}=\dfrac{1}{i!j!}\dfrac{%
\partial ^{n}Q}{\partial y^{i}\partial x^{j}}\left( a,b\right) 
\end{equation*}
By assumption, we have 
\begin{equation*}
a_{0,0}=0,\;\;a_{0,1}=0,\;\;a_{0,2}\neq 0
\end{equation*}
thus 
\begin{equation*}
a_{0,1}=b\sum_{n=1}^{N}nA_{n}a^{n-1}b^{\left( n^{2}+n-2\right) /2}=0
\end{equation*}
and since $b\neq 0$%
\begin{equation}
A_{1}=-\sum_{n=2}^{N}nA_{n}a^{n-1}b^{\left( n^{2}+n-2\right) /2}\text{.}
\end{equation}
Then

\begin{equation*}
a_{1,0}=\dfrac{\partial Q}{\partial y}\left( a,b\right) =\sum_{n=1}^{N}%
\dfrac{n\left( n+1\right) }{2}A_{n}a^{n}b^{\left( n^{2}+n-2\right) /2}
\end{equation*}
or 
\begin{equation*}
a_{1,0}=aA_{1}+\sum_{n=2}^{N}\dfrac{n\left( n+1\right) }{2}%
A_{n}a^{n}b^{\left( n^{2}+n-2\right) /2}
\end{equation*}
and with $\left( 27\right) $%
\begin{equation*}
a_{1,0}=\sum_{n=2}^{N}\dfrac{n\left( n-1\right) }{2}A_{n}a^{n}b^{\left(
n^{2}+n-2\right) /2}\text{.}
\end{equation*}
In addition, we have 
\begin{equation*}
a_{0,2}=\dfrac{1}{2}\sum_{n=2}^{N}n\left( n-1\right) A_{n}a^{n-2}b^{n\left(
n+1\right) /2}\neq 0
\end{equation*}
and since $a\neq 0$, we get 
\begin{equation*}
a_{1,0}=\dfrac{a^{2}}{b}a_{0,2}\neq 0\text{.}
\end{equation*}
The Newton polygon is thus reduced in a $\left( h,k\right) $-plane, to the
segment that connects the point $\left( 0,2\right) $ to point $\left(
1,0\right) $. The Puiseux theorem asserts the existence of a convergent
Puiseux series $\psi \left( h\right) $, satisfying $T\left( h,\psi \left(
h\right) \right) =0$. Let us determine the first term of this series with
the method of Newton-Puiseux (see for exemple $\left[ 10\right] $ Chap. 7).
The quasi-homogeneous polynomial corresponding to the Newton polygon is 
\begin{equation*}
\widetilde{T}\left( h,k\right) =\left( \dfrac{a^{2}}{b}h+k^{2}\right) a_{0,2}
\end{equation*}
So, the parametrization is at the first approximation 
\begin{equation*}
\left\{ 
\begin{array}{c}
h=t^{2} \\ 
k=\lambda t
\end{array}
\right. 
\end{equation*}
where $\lambda $ is determined by the equation $\widetilde{T}\left(
h,k\right) =\left( \dfrac{a^{2}}{b}+\lambda ^{2}\right) t^{2}a_{0,2}=0$. It
follows $\lambda =\left( -a^{2}/b\right) ^{1/2}$, giving the first term $%
\left( 26\right) $.
\end{proof}

\bigskip

\begin{proof}[Proof of Theorem 6.1]
\bigskip We are going to show that there are no zeros of the function $%
B_{y}\left( f\right) $, which may have an order greater than one for $y\in %
\left] 0,1\right[ $. Let $m$ be a natural number superior or equal to $2$
and $f$ $\left( x\right) =\sum_{n=0}^{N}A_{n}x^{n}$ be a $L_{2}$ function.
Here, $N$ is an integer superior or equal to $m$ if $f$ is a polynomial of
degree $N$, or $N=+\infty $ if $f$ is transcendent. Assume the existence of $%
b\in \left] 0,1\right[ $ and $a\in \mathbb{R}$, such that $a$ is a zero of $%
B_{b}\left( f\right) $ with order $m$, ie 
\begin{equation*}
\begin{tabular}{lllll}
$B_{b}\left( f\right) \left( a\right) $ & $=$ & $Q\left( a,b\right) $ & $=$
& $0$ \\ 
$\left( B_{b}\left( f\right) \right) ^{\prime }\left( a\right) $ & $=$ & $%
\dfrac{\partial Q}{\partial x}\left( a,b\right) $ & $=$ & $0$ \\ 
$-----$ & $-$ & $------$ & $-$ & $-$ \\ 
$\left( B_{b}\left( f\right) \right) ^{\left[ m-1\right] }\left( a\right) $
& $=$ & $\dfrac{\partial ^{m-1}Q}{\partial x^{m-1}}\left( a,b\right) $ & $=$
& $0$ \\ 
$\left( B_{b}\left( f\right) \right) ^{\left[ m\right] }\left( a\right) $ & $%
=$ & $\dfrac{\partial ^{m}Q}{\partial x^{m}}\left( a,b\right) $ & $\neq $ & $%
0$%
\end{tabular}
\end{equation*}
Note that we have $a\neq 0$, otherwise $B_{b}\left( f\right) \left( 0\right)
=A_{0}b=0$ with $b\neq 0$, which implies $A_{0}=f\left( 0\right) =0$ in
contradiction with the hypothesis. An easy induction shows that we have
formally for all $l\in \mathbb{N}$ 
\begin{equation*}
B_{y}\left( f^{\left[ l\right] }\right) \left( x\right) =y^{-l\left(
l+1\right) /2}\left( B_{y}\left( f\right) \right) ^{\left[ l\right] }\left( 
\dfrac{x}{y^{l}}\right) 
\end{equation*}
In particular, we have

\begin{equation*}
B_{b}\left( f^{\left[ m-2\right] }\right) \left( x\right) =b^{-\left(
m-2\right) \left( m-1\right) /2}\left( B_{b}\left( f\right) \right) ^{\left[
m-2\right] }\left( \dfrac{x}{b^{m-2}}\right) 
\end{equation*}
where the power series in $x$\ are convergent on $\mathbb{C}$ if $N=+\infty $%
. Therefore 
\begin{equation*}
\left( B_{b}\left( f^{\left[ m-2\right] }\right) \right) ^{\prime }\left(
x\right) =b^{-\left( m-2\right) \left( m+1\right) /2}\left( B_{b}\left(
f\right) \right) ^{\left[ m-1\right] }\left( \dfrac{x}{b^{m-2}}\right) 
\end{equation*}
and 
\begin{equation*}
\left( B_{b}\left( f^{\left[ m-2\right] }\right) \right) ^{\prime \prime
}\left( x\right) =b^{-\left( m-2\right) \left( m+3\right) /2}\left(
B_{b}\left( f\right) \right) ^{\left[ m\right] }\left( \dfrac{x}{b^{m-2}}%
\right) 
\end{equation*}
We set $g=f^{\left[ m-2\right] }$ and $c=ab^{m-2}$. The degree of $g$ is\
superior to 1 or $g$ is transcendent. $c$ is a real number different from
zero and we have 
\begin{equation*}
\begin{tabular}{lllll}
$B_{b}\left( g\right) \left( c\right) $ & $=$ & $b^{-\left( m-2\right)
\left( m-1\right) /2}\left( B_{b}\left( f\right) \right) ^{\left[ m-2\right]
}\left( a\right) $ & $=$ & $0$ \\ 
$\left( B_{b}\left( g\right) \right) ^{\prime }\left( c\right) $ & $=$ & $%
b^{-\left( m-2\right) \left( m+1\right) /2}\left( B_{b}\left( f\right)
\right) ^{\left[ m-1\right] }\left( a\right) $ & $=$ & $0$ \\ 
$\left( B_{b}\left( g\right) \right) ^{\prime \prime }\left( c\right) $ & $=$
& $b^{-\left( m-2\right) \left( m+3\right) /2}\left( B_{b}\left( f\right)
\right) ^{\left[ m\right] }\left( a\right) $ & $\neq $ & $0$%
\end{tabular}
\end{equation*}
It follows that $B_{y}\left( g\right) $ satisfies the assumptions of Lemma
6.3 with $\left( c,b\right) \in \mathbb{C}^{2}$, where $c$ and $b$ are
different from $0$. Hence, there exists a convergent Puiseux series $k=\psi
\left( h\right) ,$ such that $B_{b+h}\left( g\right) \left( c+\psi \left(
h\right) \right) =0$, and whose the first term is 
\begin{equation*}
\left( \dfrac{-c^{2}}{b}h\right) ^{1/2}
\end{equation*}
It is possible to choose a determination of the power $z^{1/2}$, holomorphic
in an open subset containing $\mathbb{R}_{-}^{\ast }$ and such that $\left(
-1\right) ^{1/2}=i$. So, we have for $h>0$ and small enough 
\begin{equation*}
\psi \left( h\right) =i\dfrac{\left| c\right| }{\sqrt{b}}\sqrt{h}+o\left( 
\sqrt{h}\right) \text{.}
\end{equation*}
It follows that for positive small enough $h$, $\psi \left( h\right) $ will
have a nonzero imaginary part. It is the same for $c+\psi \left( h\right) $
since $c$ is a real number. We just prove that there is $\varepsilon >0$
such that $B_{y}\left( g\right) $ has at least one non-real zero, for $y\in 
\left] b,b+\varepsilon \right[ \subset \left] 0,1\right[ $. But the last
point of Proposition 3.1 shows that $g$ is $L_{2}$. Thus, we have a
contradiction because the point $i)$ of Theorem 4.1 shows that $B_{y}\left(
g\right) $ is also $L_{2}$, and has therefore only real zeros, for $y\in 
\left[ 0,1\right] $.
\end{proof}

\begin{remark}
\bigskip Theorem 6.1 can be used in the proof of Lemma 5.3 instead of the
point 3 of Theorem 3.3. It can also be used to prove with Theorem 5.1, that
RH implies actually that $\Delta _{n}\left( y\right) $ does not vanish on
the interval $\left[ 0,1\right[ $ for all $n\geq 0$.
\end{remark}

It is natural to think that the reciprocal of Theorem 6.1 is true. So, we
formulate the following conjecture.

\bigskip

\textbf{Conjecture} \textit{If }$f$\textit{\ is an entire real function
whithout gaps (in the sense of Definition 4) and }$\Omega =\sup_{n\geq
1}\left| \Omega _{n}\right| <+\infty $\textit{, then }$f$\textit{\ is }$%
L_{2} $\textit{\ with }$f\left( 0\right) \neq 0$, \textit{is equivalent to
that the zeros of }$B_{y}\left( f\right) $\textit{\ are simple for }$y\in %
\left] 0,1\right[ $\textit{.}

\bigskip

First, note that with the function $f^{\flat }$, the conjecture gives the
reciprocal of the final implication\ of Corollary 6.2, ie the conjecture
announced in the abstract

We just prove the necessity of the condition of simplicity of the zeros. We
now discuss the sufficiency of this condition.

For each $p\geq 1$ we can write the power series (see $\left( 8\right) $ et $%
\left( 9\right) $ above)

\begin{equation*}
\alpha _{p}\left( y\right) =\dfrac{A_{p}}{A_{p-1}}y^{p}\left( 1+\dfrac{%
A_{p-1}}{A_{p}}\sum\limits_{q=1}^{\infty }u_{p}\left( q\right) y^{q}\right)
\end{equation*}
where the power series in parentheses converges for $\left| y\right| <\rho
_{o}^{2}\Omega ^{-1}$ (as before, we assume $\rho _{o}^{2}\Omega ^{-1}<1$,
otherwise we have already seen that $f$ is $L_{2}$ without requirement of
simplicity of zeros). It follows that $\alpha _{p}\left( y\right) $ does not
vanish in a punctured disk $0<\left| y\right| <R_{p}\leq \rho _{o}^{2}\Omega
^{-1}$, where $R_{p}$ depends a priori on $p$. Each zeros of $B_{y}\left(
f\right) $, $x_{p}\left( y\right) =-\left( \alpha _{p}\left( y\right)
\right) ^{-1}$, is therefore a holomorphic function in this punctured disk.
We can now try to follow the reasoning of the point $ii)$ of Lemma 5.2, by
using the implicit function theorem. For each $p\geq 1$, this would lead to
an analytic continuation $\widehat{x_{p}}\left( y\right) $\ of the
holomorphic function $x_{p}\left( y\right) $, where $\widehat{x_{p}}\left(
y\right) $ is now defined in an open subsets containing $\left] 0,1\right[ $%
. The rest of the reasoning would be the same as that of Lemma 5.2. However,
there is now a challenge to make this reasoning rigorous, because we can not
exclude here the fact that some zeros of $B_{y}\left( f\right) $ tend to
infinity for certain values of $y$\ in the interval $\left] 0,1\right[ $.

\textbf{Acknowledgements :} I am grateful to H.Queff\'{e}lec for advices. I
also thank P.D\`{e}bes for useful discussions we had.

\bigskip

\textbf{References}

$\bigskip $

$\left[ 1\right] $ M. Balazard, Un si\`{e}cle et demi de recherches sur
l'hypoth\`{e}se de Riemann, Gaz. Math.126 (2010), p. 7-24.

$\left[ 2\right] $ J.P. Bezivin, Sur les \'{e}quations fonctionnelles aux
q-diff\'{e}rences, Aequationes Math. 43 (2-3), (1992), p.159-176.

$\left[ 3\right] $ V. Brugidou, A new method to determine the value or the
reality of zeros for certain entire functions, J.Math.Pures Appl. 94, Nb.3
(2010), p. 244-276.

$\left[ 4\right] $ V. Brugidou, Une g\'{e}n\'{e}ralisation de la formule du
triple produit de Jacobi et quelques applications, C. R. A. S. 349 (2011),
p. 357-484.

$\left[ 5\right] $ V.Brugidou, private communication to Paul
Malliavin (July 1, 2008) and others; in preparation for publication.

$\left[ 6\right] $ N.G. de Bruijn, The roots of trigonometric integrals,
Duke J. Math.17, (1950), p. 197-226.

$\left[ 7\right] $ T. Craven,G. Csordas, Composition theorems, multiplier
sequences and complex zero decreasing sequences, in: Value Distribution
Theory and Its Related Topics, G. Barsegian, I. Laine, C.C. Yang (Eds.),
Kluwer Academic, Dordrecht, 2004.

$\left[ 8\right] $ G. Csordas, T.S. Norfolk, R.S. Varga, The Riemann
hypothesis and the Turan inequalities, Trans.of the Amer.Math.Soc. 296, Nb.2
(1986), p. 521-541.

$\left[ 9\right] $ H.M. Edwards, Riemann's Zeta Function, Academic Press,
1974.

$\left[ 10\right] $ G. Fischer, Plane algebraic curves, AMS, Student
mathematical library, 2001.

$\left[ 11\right] $ H. Hille, Analytic function theory, vol. II, Ginn and
Co., Boston, 1962.

$\left[ 12\right] $ L. Iliev, Laguerre entire functions, Publishing house of
the Bulg. Acad. of Sciences, Sofia, 1987.

$\left[ 13\right] $ C.M. Newman, Fourier Transforms with only real zeros,
Proc. Amer. Math. Soc., 61 (1976), p. 245-251.

$\left[ 14\right] $ G. Polya, \"{U}ber die
algebraisch-funktionentheoretischen Untersuchungen von J.L.W.V. Jensen, Klg.
Danske Vid. Sel. Math-Fys. Medd. 7 (1927), p. 224-249.

$\left[ 15\right] $ E.C. Titchmarsh, The Theory of the Riemann
Zeta-Function, 2$%
{{}^\circ}%
$ ed., Oxford Press, 1986.

$\left[ 16\right] $ L.Di.Vizio, J.P. Ramis, J. Sauloy, C. Zhang, Equations
aux $q$-diff\'{e}rences, Gaz. Math. 96 (2003), p. 20-49.

$\left[ 17\right] $ http://www.aimath.org/pl/rhequivalences

\end{document}